\documentclass[a4paper]{amsart}

\usepackage[english]{babel}
\usepackage{amsfonts}
\usepackage{amsmath}
\usepackage{amssymb}
\usepackage{amsthm}
\usepackage{latexsym}
\usepackage{amstext}
\usepackage{enumerate}
\usepackage{graphics}
\usepackage[all]{xy}
\usepackage[pdftex,colorlinks]{hyperref}

\theoremstyle{plain}
\newtheorem{theorem}{Theorem}[section]

\newtheorem{corollary}[theorem]{Corollary}
\newtheorem{lemma}[theorem]{Lemma}
\newtheorem{proposition}[theorem]{Proposition}

\theoremstyle{definition}
\newtheorem{definition}[theorem]{Definition}

\newtheorem{remark}[theorem]{Remark}
\numberwithin{equation}{section}

\newcommand{\R}{\mathbb{R}}

\newcommand{\Sf}{\mathbb{S}}

\begin{document}

\title[Curvature properties of $\varphi$-null Osserman Lorentzian $\mathcal{S}$-manifolds]
{Curvature properties of $\varphi$-null Osserman Lorentzian $\mathcal{S}$-manifolds}
\author{Letizia Brunetti and Angelo V. Caldarella}

\address{Department of Mathematics, University of Bari\newline%
\indent Via E.~Orabona, 4 -- I 70125 Bari (Italy)}%
\email[L.~Brunetti]{brunetti@dm.uniba.it}%
\email[A.V.~Caldarella]{caldarella@dm.uniba.it}%

\date{\today}
\subjclass{53C50, 53C25, 53B30}
\keywords{Lorentz manifold. Semi-Riemannian manifold. Jacobi operator. Osserman condition. Indefinite $\mathcal{S}$-manifold.}%

\begin{abstract}
We expound some results about the relationships between the Jacobi operators with respect to null vectors on a Lorentzian $\mathcal{S}$-manifold $M$ and the Jacobi operators with respect to particular spacelike unit vectors on $M$. We study the number of the eigenvalues of such operators in a $\varphi$-null Osserman Lorentzian $\mathcal{S}$-manifold, under suitable assumptions on the dimension of the manifold. Then, we generalize a curvature characterization, previously obtained by the first author for Lorentzian $\varphi$-null Osserman $\mathcal{S}$-manifolds with exactly two characteristic vector fields, to the case of those with an arbitrary number of characteristic vector fields.
\end{abstract}

\maketitle

\section{Introduction}
The Jacobi operator is one of the main objects of study in Riemannian and semi-Riemannian Geometry, due to the consequences of its behaviour on several geometrical properties of a manifold.

Let $(M,g)$ be an $n$-dimensional semi-Riemannian manifold of signature $(h,k)$, with $h+k=n$, where we first put the plus signs and then the minus signs. If $k=0$ then $M$ is a Riemannian manifold, while if $k=1$, $M$ is Lorentzian. For any $p\in M$, the unit spacelike (resp.\ timelike) tangent sphere at $p$ is the set $S^+_p(M)=\{z\in T_pM\,|\,g_p(z,z)=+1\}$ (resp.\ $S^-_p(M)=\{z\in T_pM\,|\,g_p(z,z)=-1\}$), and the unit spacelike (resp.\ timelike) sphere bundle is $S^+(M)=\bigcup_{p\in M}S^+_p(M)$ (resp.\ $S^-(M)=\bigcup_{p\in M}S^-_p(M)$). We put $S_p(M)=S^+_p(M)\cup S^-_p(M)$ and $S(M)=\bigcup_{p\in M}S_p(M)$. Note that for a Riemannian manifold one has $S_p(M)=S^+_p(M)$ and $S(M)=S^+(M)$.

For any $z\in S_p(M)$, $p\in M$, the \emph{Jacobi operator with respect to $z$} is the endomorphism $R_z\colon z^{\bot}\rightarrow z^{\bot}$ defined by $R_z(\cdot)=R_p(\cdot,z)z$ (see, for example, \cite{GKV}), where $R$ is the $(1,3)$-type curvature tensor field of $(M,g)$.

The Jacobi operator is a self-adjoint map, thus it is diagonalizable in the Riemannian case, and the study of its eigenvalues has a special interest: for example, they indicate the extreme values of the sectional curvatures of the manifold. The Jacobi operator is involved in geodesic deformations and it plays a central role in the Jacobi equations. It is worth noting that, in the Riemannian case, the eigenvalues of the Jacobi operators $R_z$ depend both on the vector $z\in S_p(M)$ and on the point $p\in M$. On the other side, in the semi-Riemannian context, dealing with the coefficients of the characteristic polynomial of $R_z$ is better than dealing with its eigenvalues, due to the diagonalization problems in the indefinite case. It is easy to see that a Riemannian manifold has constant sectional curvature $c$ if and only if the Jacobi operators $R_z$ have exactly one constant eigenvalue $\lambda=c$, for any $z\in S(M)$. Then it appears natural to study what happens if the Jacobi operators of a Riemannian manifold admit eigenvalues independent both of the vector $z\in S_p(M)$ and of the point $p\in M$, that is when $(M,g)$ is an \emph{Osserman manifold}. In the indefinite setting, a semi-Riemannian manifold $(M,g)$ is said to be \emph{spacelike} (resp.\ \emph{timelike}) \emph{Osserman}, if the characteristic polynomial of $R_z$ is independent both of $z\in S^+_p(M)$ (resp.\ $z\in S^-_p(M)$) and $p\in M$. In \cite{GKVV} it is proved that $(M,g)$ being spacelike Osserman is equivalent to $(M,g)$ being timelike Osserman. For a general account on the geometry of Osserman manifolds, we refer the reader to \cite{GKV}.

It is known that any locally flat or locally rank-one symmetric space is an Osserman manifold, while the converse statement is known as the Osserman Conjecture, proposed by R.~Osserman in \cite{OS} (see also \cite{Oss}). Several results have been obtained in search of the solution of the Conjecture: in the Riemannian setting it was proved by Q.S.~Chi in many cases (\cite{C0}, \cite{C1}, \cite{C2}), while more recently Y.~Nikolayevsky provided results for almost any case which is not covered by Q.S.~Chi (see \cite{N01}, \cite{N02}, \cite{N03}).

The Osserman problem was also considered in the Lorentzian setting by E.~Garc\'ia-R\'io, D.N.~Kupeli and M.E.~V\'azquez-Abal (\cite{GK}, \cite{GKVA}), who together with N.~Bla\v{z}i\'c, N.~Bokan and P.~Gilkey (\cite{BBG}) gave a complete affirmative answer to the Osserman Conjecture, by proving that a Lorentzian manifold is Osserman if and only if it has constant sectional curvature. They first considered the spacelike and timelike cases separately, and now a new and simple proof of both cases is provided in \cite{GKV}.

In the case of metrics with indefinite arbitrary signature, there are several counterexamples to the Conjecture (see for example \cite{BBR}, \cite{BCGHV}, \cite{GVV}).

Very recently, some new conditions of Osserman type have been introduced and studied in \cite{AD} and \cite{ALT}, where the authors, based on \cite{AET}, deal with suitable (pseudo) Jacobi operators associated with the degenerate metric induced on lightlike hypersurfaces and lightlike submanifolds of semi-Riemannian manifolds.

Now, we concentrate our interest in the Lorentzian case, where the Osserman problem is completely solved. In \cite{GKVA} the attention has been focused on the Osserman conditions related with null (lightlike) vectors. In that paper, the authors first introduce the Jacobi operator $\bar{R}_u$ with respect to a null vector $u$ and then they define and study the so-called \emph{null Osserman condition} with respect to a unit timelike vector (see also \cite{GKV}).

Lorentzian almost contact manifolds are also interesting in relation to the null Osserman conditions. It is known that any Lorentz Sasakian manifold $(M,\varphi,\xi,\eta,g)$ is globally null Osserman with respect to the timelike characteristic vector field $\xi$, if the manifold has constant $\varphi$-sectional curvature. Yet, as seen in \cite{BR}, a similar result fails when considering Lorentzian $\mathcal{S}$-manifolds with constant $\varphi$-sectional curvature. In Section $3$ we prove that, more generally, no Lorentzian $\mathcal{S}$-manifold can be null Osserman, removing the hypothesis of constant $\varphi$-sectional curvature. On the other side, a Lorentzian $\mathcal{S}$-manifold can never be Osserman, because such manifolds do not have constant sectional curvature. This has led the first author to introduce and study in \cite{BR} a new kind of null Osserman condition, called \emph{$\varphi$-null Osserman condition}, which seems to be more convenient for manifolds carrying Lorentzian globally framed $f$-structures. It reduces to the classical null Osserman condition when considering Lorentzian almost contact structures.

In this paper we are going to proceed further in this direction. As main results, in Section $3$, we obtain a link between the behaviour of the Jacobi operators with respect to null vectors on a Lorentzian $\mathcal{S}$-manifold $(M,\varphi,\xi_\alpha,\eta^\alpha,g)$ and the Jacobi operators with respect to unit vectors in $\mathrm{Im}(\varphi)$. Then we use it to prove a result on the number of the eigenvalues of these operators for a $\varphi$-null Osserman Lorentzian $\mathcal{S}$-manifold, under suitable assumptions on the dimension of the manifold, in analogy to some results obtained by Q.S.~Chi (\cite{C0}) for Osserman Riemannian manifolds. In Section $4$ we give an extension of the curvature results of \cite{BR} to the case of an arbitrary $\varphi$-null Osserman Lorentzian $\mathcal{S}$-manifold. We first provide a remarkable characterization of a suitable class of curvature-like maps on a Lorentzian $g.f.f$-manifold, and then, using this result, we prove a curvature characterization for $\varphi$-null Osserman Lorentzian $\mathcal{S}$-manifolds with an arbitrary number of characteristic vector fields.

In what follows, all manifolds, tensor fields and maps are assumed to be smooth. Moreover, all manifolds are supposed to be connected and, according to \cite{KN}, for the Riemannian curvature tensor $R$ of a semi-Riemannian manifold $(M,g)$ we put
\[
R(X,Y,Z,W)=g(R(Z,W)Y,X)=g(([\nabla_Z,\nabla_W]-\nabla_{[Z,W]})Y,X),
\]
for any vector fields $X,Y,Z,W$ on $M$. It is well-known that the following fundamental symmetries hold
\[
\begin{split}
&R(X,Y,Z,W)=R(Z,W,X,Y),\\
&R(X,Y,Z,W)=-R(Y,X,Z,W)=-R(X,Y,W,Z),\\
&R(X,Y,Z,W)+R(Y,Z,X,W)+R(Z,X,Y,W),
\end{split}
\]
for any $X,Y,Z,W\in\Gamma(TM)$. More generally, if $p\in M$, any multilinear map $F\colon T_p(M)^4\rightarrow\R$ is said to be a \emph{curvature-like map} if it satisfies the above symmetries (\cite{O'N}). For any linearly independent vectors $x,y\in T_pM$ spanning a non-degenerate plane $\pi=\mathrm{span}(x,y)$, that is $\Delta(\pi)=g_p(x,x)g_p(y,y)-g_p(x,y)^2\neq0$, the \emph{sectional curvature of $F$ with respect to $\pi$} is, by definition, the real number
\[
K(\pi)=K(x,y)=\frac{F(x,y,x,y)}{\Delta(\pi)}.
\]
It is well-known that the sectional curvature $K(\pi)$ is independent of the non-degenerate plane $\pi$ if, and only if, $F(x,y,z,w)=k(g(x,z)g(y,w)-g(y,z)g(x,w))$, for all $x,y,z,w\in T_p(M)$, with $K(\pi)=k\in\R$ (\cite[p. 80]{O'N}). More generally, a special feature of semi-Riemannian manifolds is that the sectional curvature on non-degenerate planes can be linked to the behaviour of the curvature with respect to degenerate planes, as provided in \cite[Thm.\ 5]{DN} and \cite[p.\ 229]{O'N}. Since the proof of the cited results only involves the algebraic symmetries of the Riemannian curvature tensor, we report the result here, for later use, stated for any curvature-like map on a Lorentzian manifold.

\begin{lemma}\label{degenere}
Let $(M,g)$ be a Lorentzian manifold and $F\colon(T_pM)^4\rightarrow\R$ a curvature-like map, $p\in M$. The following conditions are equivalent.
\begin{itemize}
	\item[(a)] $F(x,y,z,w)=k(g(x,z)g(y,w)-g(y,z)g(x,w))$, for all $x,y,z,w\in T_p(M)$, with $k\in\R$;
	\item[(b)] $F(x,y,x,y)=0$, for any degenerate plane $\pi=span\{x,y\}$ in $T_p(M)$.
\end{itemize}
\end{lemma}

\section{Preliminaries}
Let us recall some basic definitions and facts about contact and $\mathcal{S}$-structures which we will need in the rest of the paper.

Following \cite{Bl1}, an \emph{almost contact structure} on a $(2n+1)$-dimensional manifold $M$ is, by definition, a triple $(\varphi,\xi,\eta)$, where $\varphi$ is a $(1,1)$-type tensor field on $M$, $\xi$ a vector field and $\eta$ a $1$-form such that $\varphi^2=-I+\eta\otimes\xi$ and $\eta(\xi)=1$, where $I\colon TM\rightarrow TM$ is the identity mapping. From the previous conditions one deduces that $\varphi\xi=0$ and $\eta\circ\varphi=0$. Moreover, the endomorphism $\varphi$ has constant rank $2n$ and, for any $p\in M$, one has the splitting $T_pM=\mathrm{Im}(\varphi_p)\oplus\mathrm{span}(\xi_p)$. The condition $\eta=0$ defines the $2n$-dimensional non-integrable distribution $\ker(\eta)=\mathrm{Im}(\varphi)$, called the \emph{contact distribution}, while $\xi$ is called the \emph{characteristic vector field} of the almost contact structure. An almost contact manifold $(M,\varphi,\xi,\eta)$ is said to be \emph{normal} if the $(1,2)$-type tensor field $N=[\varphi,\varphi]+2d\eta\otimes\xi$ vanishes identically, where $[\varphi,\varphi]$ is the Nijenhuis tensor field of $\varphi$, defined by $[\varphi,\varphi](X,Y)=\varphi^2[X,Y]+[\varphi X,\varphi Y]-\varphi[\varphi X,Y]-\varphi[X,\varphi Y]$, for any $X,Y\in\Gamma(TM)$.

An indefinite metric tensor $g$ on an almost contact manifold $(M,\varphi,\xi,\eta)$ is said to be \emph{compatible} with the almost contact structure $(\varphi,\xi,\eta)$ if
\begin{equation}\label{eq:01}
	g(\varphi X,\varphi Y) = g(X,Y)-\varepsilon\eta(X)\eta(Y),
\end{equation}
for all $X,Y\in\Gamma(TM)$, where $\varepsilon=g(\xi,\xi)=\pm1$. Then, the manifold $M$ is said to be an \emph{indefinite almost contact metric manifold} with structure $(\varphi,\xi,\eta,g)$. From (\ref{eq:01}) it follows easily that $g(X,\xi)=\varepsilon\eta(X)$ and $g(X,\varphi Y)=-g(\varphi X,Y)$, for any $X,Y\in\Gamma(TM)$, as well as that $\mathrm{Im}(\varphi)$ is orthogonal to the distribution $\langle\xi\rangle$ spanned by $\xi$. The $2$-form $\Phi$ on $M$ defined by $\Phi(X,Y)=g(X,\varphi Y)$ is called the \emph{fundamental}, or the \emph{Sasakian $2$-form} of the indefinite almost contact metric manifold. If $\Phi=d\eta$, the manifold $(M,\varphi,\xi,\eta,g)$ is said to be an \emph{indefinite contact metric manifold}. Finally, a normal indefinite contact metric manifold is, by definition, an \emph{indefinite Sasakian manifold}. It is known that an indefinite almost contact metric manifold is indefinite Sasakian if and only if the covariant derivative of $\varphi$ satifies $(\nabla_X\varphi)Y=g(X,Y)\xi-\varepsilon\eta(Y)X$, with $\varepsilon=g(\xi,\xi)=\pm1$. It follows easily that $\nabla_X\xi=-\varepsilon\varphi X$ and that $\xi$ is a Killing vector field. Standard reference for contact structures in the Riemannian case is \cite{Bl1}, while for the indefinite case the reader is referred to \cite{D} and \cite{T}.

In \cite{Na01} and \cite{Na02}, H.~Nakagawa introduced a generalization of the above structures with the notion of framed $f$-manifold, later developed and studied by S.I.~Goldberg and K.~Yano (\cite{GY}, \cite{GY02}) and others with the denomination of globally framed $f$-manifolds. A manifold $M$ is said to be a \emph{globally framed $f$-manifold} (briefly \emph{$g.f.f$-manifold}) if it carries a globally framed $f$-structure, that is a non-vanishing $(1,1)$-type tensor field $\varphi$ on $M$ of constant rank satisfying $\varphi^3+\varphi=0$, and such that the subbundle associated with the distribution $\ker(\varphi)$ is parallelizable. If $\dim(\ker(\varphi))=s\geqslant1$, this is equivalent to the existence of $s$ linearly independent vector fields $\xi_\alpha$ and $1$-forms $\eta^\alpha$, $\alpha\in\{1,\ldots,s\}$, such that
\begin{equation}\label{eq:02}
\varphi^2=-I+\eta^\alpha\otimes\xi_\alpha\qquad\text{and}\qquad\eta^\alpha(\xi_\beta)=\delta^\alpha_\beta,
\end{equation}
where $I$ is the identity mapping. Clearly, if $s=1$ we get an almost contact structure. From (\ref{eq:02}) it follows that $\varphi\xi_\alpha=0$ and $\eta^\alpha\circ\varphi=0$, for any $\alpha\in\{1,\ldots,s\}$. The conditions $\eta^1=\dots=\eta^s=0$ define the $2n$-dimensional distribution $\mathrm{Im}(\varphi)$ on which $\varphi$ acts as an almost complex tensor field. One has the splitting $T_pM=\mathrm{Im}(\varphi_p)\oplus\mathrm{span}((\xi_1)_p,\ldots,(\xi_s)_p)$, for any $p\in M$, and $\dim(M)=2n+s$. Each $\xi_\alpha$ is said to be a \emph{characteristic vector field} of the structure. A $g.f.f$-manifold $(M,\varphi,\xi_\alpha,\eta^\alpha)$, $\alpha\in\{1,\ldots,s\}$, is called \emph{normal} if the $(1,2)$-type tensor field $N=[\varphi,\varphi]+2d\eta^\alpha\otimes\xi_\alpha$ vanishes identically.

From now on, a $(2n+s)$-dimensional $g.f.f$-manifold will be simply denoted by $(M,\varphi,\xi_\alpha,\eta^\alpha)$, leaving the condition $\alpha\in\{1,\ldots,s\}$ understood.

Following \cite{LP}, an indefinite metric $g$ on a $g.f.f$-manifold $(M,\varphi,\xi_\alpha,\eta^\alpha)$ is said to be \emph{compatible} with the $g.f.f$-structure $(\varphi,\xi_\alpha,\eta^\alpha)$ if
\begin{equation}\label{eq:03}
	g(\varphi X,\varphi Y) = g(X,Y)-\sum_{\alpha=1}^s\varepsilon_\alpha\eta^\alpha(X)\eta^\alpha(Y),
\end{equation}
for all $X,Y\in\Gamma(TM)$, where $\varepsilon_\alpha=g(\xi_\alpha,\xi_\alpha)=\pm1$. Then, the manifold $M$ is said to be an \emph{indefinite metric $g.f.f$-manifold} with structure $(\varphi,\xi_\alpha,\eta^\alpha,g)$. From (\ref{eq:03}) it follows that $g(X,\xi_\alpha)=\varepsilon_\alpha\eta^\alpha(X)$ and $g(X,\varphi Y)=-g(\varphi X,Y)$, for any $X,Y\in\Gamma(TM)$, as well as that $\mathrm{Im}(\varphi)$ is orthogonal to the distribution  $\langle\xi_1,\ldots,\xi_s\rangle$ spanned by the vector fields $\xi_\alpha$, $\alpha\in\{1,\ldots,s\}$. The $2$-form $\Phi$ on $M$ defined by $\Phi(X,Y)=g(X,\varphi Y)$ is called the \emph{fundamental $2$-form} of the indefinite metric $g.f.f$-manifold. If $\Phi=d\eta^\alpha$, for any $\alpha\in\{1,\ldots,s\}$, the manifold $(M,\varphi,\xi_\alpha,\eta^\alpha,g)$ is said to be an \emph{indefinite almost $\mathcal{S}$-manifold}. Finally, a normal indefinite almost $\mathcal{S}$-manifold is, by definition, an \emph{indefinite $\mathcal{S}$-manifold}. In \cite{LP} it is proved that in an indefinite $\mathcal{S}$-manifold the covariant derivative of $\varphi$ satisfies $(\nabla_X\varphi)Y=g(X,Y)\tilde\xi+\tilde\eta(Y)\varphi^2X$, where $\tilde\xi=\sum_{\alpha=1}^s\xi_\alpha$ and $\tilde\eta=\sum_{\alpha=1}^s\varepsilon_\alpha\eta^\alpha$. It follows that, for any $\alpha,\beta\in\{1,\ldots,s\}$,
\begin{equation}\label{eq:16}
\nabla_X\xi_\alpha=-\varepsilon_\alpha\varphi X \qquad \text{and} \qquad \nabla_{\xi_\alpha}\xi_\beta=0,
\end{equation}
as well as that each $\xi_\alpha$ is a Killing vector field. In particular, for $s=1$ one finds again the notion of indefinite Sasakian manifold (\cite{T}).

For further properties on $\mathcal{S}$-manifolds, in the Riemannian context, we refer the reader to \cite{Bl01}, \cite{BLY}, \cite{CFF} and \cite{DIP}, where the notion of almost $\mathcal{S}$-manifold is introduced, while the generalization to the semi-Riemannian setting is given in \cite{LP}.

We have the following useful formulas for the Riemannian curvature tensor of an indefinite $\mathcal{S}$-manifold.

\begin{lemma}\label{lem:99}
Let $(M,\varphi,\xi^\alpha,\eta_\alpha,g)$ be a $(2n+s)$-dimensional indefinite $\mathcal{S}$-manifold, $s\geqslant 1$. The following identities hold, for any $X,Y,Z\in\Gamma(TM)$, any $U,V\in\langle\xi_1,\ldots,\xi_s\rangle$ and any $\alpha,\beta,\gamma\in\{1,\ldots,s\}$.
\begin{enumerate}
	\item $R(X,Y,\xi_\alpha,Z)=\varepsilon_\alpha\left\{\tilde\eta(X)g(\varphi Y,\varphi Z)-\tilde\eta(Y)g(\varphi X,\varphi Z)\right\}$;
	\item $R(\xi_\beta,Y,\xi_\alpha,Z)=\varepsilon_\beta\varepsilon_\alpha g(\varphi Y,\varphi Z)$;
	\item $R(\xi_\beta,\xi_\gamma,\xi_\alpha,Z)=0$;
	\item $R(\varphi X,\varphi Y,\xi_\alpha,Z)=0$;
	\item $R(U,Y,V,Z)=\tilde\eta(U)\tilde\eta(V)g(\varphi Y,\varphi Z)$.
\end{enumerate}
where, $\varepsilon_\alpha=g(\xi_\alpha,\xi_\alpha)=\pm1$ for any $\alpha\in\{1,\ldots,s\}$, and $\tilde\eta=\sum_{\alpha=1}^{s}\varepsilon_\alpha\eta^\alpha$.
\end{lemma}
\proof With straightforward calculations, using (\ref{eq:16}), one gets $(1)$. The identities $(2)$, $(3)$ and $(4)$ are easy consequences of $(1)$, while $(5)$ follows from $(2)$. \endproof

In particular we have
\begin{equation}\label{eq:04}
 R(X,Y,Z,\xi_\alpha)=0 \qquad\text{and}\qquad R(X,\xi_\alpha,Y,\xi_\beta)=\varepsilon_\alpha\varepsilon_\beta g(X,Y),
\end{equation}
for any $X,Y,Z\in\mathrm{Im}(\varphi)$ and any $\alpha,\beta\in\{1,\ldots,s\}$.

Throughout the paper, given a $(2n+s)$-dimensional indefinite $g.f.f$-manifold $(M,\varphi,\xi_\alpha,\eta^\alpha,g)$, we always put $\tilde\xi=\sum_{\alpha=1}^s\xi_\alpha$ and $\tilde\eta=\sum_{\alpha=1}^s\varepsilon_\alpha\eta^\alpha$, with $\varepsilon_\alpha=g(\xi_\alpha,\xi_\alpha)=\pm1$, according to the causal character of $\xi_\alpha$.

\section{Lorentzian $\mathcal{S}$-manifolds and the $\varphi$-null Osserman condition.}
Let $(M,g)$ be a Lorentzian manifold and $p\in M$. Following \cite{GKVA} and \cite{GKV}, if $u\in T_pM$ is a lightlike (or null) vector, that is $u\neq0$ and $g_p(u,u)=0$, since $\mathrm{span}(u)\subset u^{\bot}$, we consider the quotient space $\bar{u}^{\bot}=u^{\bot}/\mathrm{span}(u)$, together with the canonical projection $\pi\colon u^{\bot}\rightarrow\bar{u}^{\bot}$. A positive definite inner product $\bar{g}$ can be defined on $\bar{u}^{\bot}$ by putting $\bar{g}(\bar{x},\bar{y})=g_p(x,y)$, where $\pi(x)=\bar{x}$ and $\pi(y)=\bar{y}$, so that $(\bar{u}^{\bot},\bar{g})$ becomes an Euclidean vector space. The \emph{Jacobi operator with respect to $\bar u$} is the endomorphism $\bar R_u\colon\bar u^\bot \rightarrow\bar u^\bot$ defined by $\bar R_u(\bar x)=\pi(R_p(x,u)u)$, for all $\bar x=\pi(x)\in\bar u^\bot$. It is easy to see that $\bar R_u$ is self-adjoint, hence diagonalizable.

Any subspace $W\subset u^\bot$ such that $u^\bot=\mathrm{span}(u)\oplus W$ is called a \emph{geometrical realization} of $\bar u^\bot$. It is a non-degenerate subspace and $\pi|_W\colon(W,g_p)\rightarrow (\bar u^\bot,\bar g)$ is an isometry, so that we can identify $(\bar u^\bot,\bar g)\cong(W,g_p)$.

If $z\in T_pM$ is a unit timelike vector, the \emph{null congruence set} of $z$ at $p$ is
\[
N(z)=\{u\in T_pM\,|\,g_p(u,u)=0,\ g_p(u,z)=-1\}.
\]
Since $z$ is timelike, the space $(T_pM,g_p)$ being Lorentzian yields $g_p(u,z)\neq0$ for any lightlike vector $u\in T_pM$, hence $N(z)$ is a non-empty set. Moreover, note that the elements of $N(z)$ are in one-to-one correspondence to those of the set $S(z)=\{v\in z^\bot\,|\,g_p(v,v)=1\}$, called the \emph{celestial sphere of $z$},
via the map $\psi\colon N(z)\rightarrow S(z)$ such that $\psi(u)=u-z$ (see \cite{GKV}).

\begin{definition}[\cite{GKVA,GKV}]
Let $(M,g)$ be a Lorentzian manifold, $p\in M$ and $z\in T_pM$ a timelike unit vector. Then, $(M,g)$ is said to be \emph{null Osserman with respect to $z$} if the eigenvalues of $\bar R_u$ and their multiplicities are independent of $u\in N(z)$.
\end{definition}

Let $(M,\varphi,\xi_\alpha,\eta^\alpha,g)$ be a Lorentzian $g.f.f$-manifold, with $\dim(M)=2n+s$, $s\geqslant1$. It is always possible to consider a local orthonormal $\varphi$-adapted frame $(X_i,\varphi X_i,\xi_\alpha)$, $1\leqslant i\leqslant n$, $1\leqslant\alpha\leqslant s$, and since $\varphi$ acts as an Hermitian structure on $\mathrm{Im}(\varphi)$, we easily see that exactly one of the characteristic vector fields has to be timelike. Thus, it is natural to study the null Osserman condition of the manifold with respect to this timelike vector. There is no loss of generality in assuming $\xi_1$ is the unit timelike vector field, and from now on, unless otherwise stated, we will always suppose that the timelike vector field of a Lorentzian $g.f.f$-manifold $(M,\varphi,\xi_\alpha,\eta^\alpha,g)$ is $\xi_1$. It is known that Lorentz Sasakian manifolds with constant $\varphi$-sectional curvature are null Osserman with respect to the characteristic vector field. We are going to see that this result is no more true when we pass to Lorentzian $\mathcal{S}$-manifolds with more than one characteristic vector field.

\begin{proposition}
Let $(M,\varphi,\xi_\alpha,\eta^\alpha,g)$ be a Lorentzian $\mathcal{S}$-manifold, with $\dim(M)=2n+s$, $s\geqslant2$. Then for any $p\in M$, the null Osserman condition with respect to $(\xi_1)_p$ does not hold.
\end{proposition}
\proof Fix $p\in M$. For brevity we drop the subscript $p$ from the notations. Let $u=\xi_1+\xi_\beta=\psi^{-1}(\xi_\beta)\in N(\xi_1)$, for some $\beta\in\{2,\ldots,s\}$, and $\bar R_u$ be the Jacobi operator. It is easy to see that
\[
W'=\mathrm{Im}(\varphi)\oplus\mathrm{span}(\xi_2,\ldots,\hat\xi_\beta,\ldots,\xi_s),
\]
is a geometrical realization of $\bar u^\bot$, where $\hat\xi_\beta$ means that the vector $\xi_\beta$ is omitted. Identifying $\bar u^\bot\cong W'$, and putting $\bar w=\pi(w)$, for any $w\in W'$, by the definition of $\bar R_u$ and using (\ref{eq:04}), if $y,z\in\mathrm{Im}(\varphi)$, we get
\[
\begin{split}
\bar g(\bar R_u(\bar y),\bar z)&=\bar g(\pi(R(y,u)u),\pi(z))=g(R(y,u)u,z)\\
	                               &=R(\xi_1,y,\xi_1,z)+R(\xi_\beta,y,\xi_\beta,z)\\
	                               &\quad+R(\xi_1,y,\xi_\beta,z)+R(\xi_\beta,y,\xi_1,z)=0;
\end{split}
\]
if $y\in\mathrm{Im}(\varphi)$ and $\gamma\in\{2,\ldots,s\}-\{\beta\}$:
\[
\begin{split}
\bar g(\bar R_u(\bar y),\bar \xi_\gamma)&=\bar g(\pi(R(y,u)u),\pi(\xi_\gamma))=g(R(y,u)u,\xi_\gamma)\\
                                        &=R(\xi_1,y,\xi_1,\xi_\gamma)+R(\xi_\beta,y,\xi_\beta,\xi_\gamma)\\
	                                      &\quad+R(\xi_1,y,\xi_\beta,\xi_\gamma)+R(\xi_\beta,y,\xi_1,\xi_\gamma)=0;
\end{split}
\]
if $\alpha,\gamma\in\{2,\ldots,s\}-\{\beta\}$:
\[
\begin{split}
\bar g(\bar R_u(\bar \xi_\alpha),\bar \xi_\gamma)&=
\bar g(\pi(R(\xi_\alpha,u)u),\pi(\xi_\gamma))=g(R(\xi_\alpha,u)u,\xi_\gamma)\\
                &=R(\xi_1,\xi_\alpha,\xi_1,\xi_\gamma)+R(\xi_\beta,\xi_\alpha,\xi_\beta,\xi_\gamma)\\
	              &\quad+R(\xi_1,\xi_\alpha,\xi_\beta,\xi_\gamma)+R(\xi_\beta,\xi_\alpha,\xi_1,\xi_\gamma)=0.
\end{split}
\]
It follows that $\bar R_u=0$. Equivalently, the only eigenvalue of $\bar R_u$ is $\lambda=0$, and assuming $M$ is null Osserman with respect to $\xi_1$, then $\bar R_u=0$ for all $u\in N(\xi_1)$. Let us choose now $v=\xi_1+x$, with $x\in\mathrm{Im}(\varphi)$, $g(x,x)=1$. Then $x\in S(\xi_1)$, and $v=\psi^{-1}(x)\in N(\xi_1)$, thus we can consider the Jacobi operator $\bar R_v$. If $V=x^\bot\cap\mathrm{Im}(\varphi)$, then
\[
W''=V\oplus\mathrm{span}(\xi_2,\ldots,\xi_s)
\]
is a geometrical realization of $\bar v^\bot$. Identifying $\bar v^\bot\cong W''$ and using again (\ref{eq:04}), for any $\alpha,\beta\in\{2,\ldots,s\}$, we obtain
\begin{equation}\label{eq:05}
\bar g(\bar R_v(\bar \xi_\alpha),\bar \xi_\beta)=R(v,\xi_\alpha,v,\xi_\beta)=R(x,\xi_\alpha,x,\xi_\beta)=1;
\end{equation}
if $y\in V$ and $\beta\in\{2,\ldots,s\}$, using also the Bianchi Identity, we have
\begin{equation}\label{eq:06}
\begin{split}
\bar g(\bar R_v(\bar y),\bar \xi_\beta)&=R(v,y,v,\xi_\beta)\\
                                       &=R(\xi_1,y,x,\xi_\beta)+R(x,y,\xi_1,\xi_\beta)+R(x,y,x,\xi_\beta)\\
                                       &=R(x,\xi_1,y,\xi_\beta)-R(x,\xi_\beta,y,\xi_1)=0.
\end{split}
\end{equation}
But (\ref{eq:05}) and (\ref{eq:06}) imply that $\bar R_v(\bar\xi_\beta)=\sum_{\alpha=2}^s\bar\xi_\alpha\neq0$, which contradicts the assumption that $M$ is null Osserman, and the claim follows. \endproof

From the above proof it is clear that the Jacobi operators $\bar R_u$, with $u=\xi_1+\xi_\beta$, have a trivial behaviour with respect to the eigenvalues, since they vanish identically. Therefore, in \cite{BR} the following new condition is introduced.

Let $(M,\varphi,\xi_\alpha,\eta^\alpha,g)$ be a Lorentz $g.f.f$-manifold, with $\dim(M)=2n+s$, $s\geqslant1$. If $p\in M$, the \emph{$\varphi$-celestial sphere} of $(\xi_1)_p$ is, by definition, the set
\[
S_\varphi((\xi_1)_p)=S((\xi_1)_p)\cap\mathrm{Im}(\varphi_p),
\]
while
\[
N_\varphi((\xi_1)_p)=\psi^{-1}(S_\varphi((\xi_1)_p)),
\]
is called the \emph{$\varphi$-null congruence set} of $(\xi_1)_p$.

\begin{definition}
Let $(M,\varphi,\xi_\alpha,\eta^\alpha,g)$ be a Lorentzian $g.f.f$-manifold, and $p\in M$. We say that $M$ is \emph{$\varphi$-null Osserman with respect to $(\xi_1)_p$} if the eigenvalues of $\bar R_u$ and their multiplicities are independent of $u\in N_\varphi((\xi_1)_p)$.
\end{definition}

One easily sees that when $M$ is a Lorentzian almost contact metric manifold $(M,\varphi,\xi,\eta,g)$, the $\varphi$-null Osserman condition reduces to the null Osserman condition, since $N_\varphi(\xi_p)=N(\xi_p)$ and $S_\varphi(\xi_p)=S(\xi_p)=\mathrm{Im}(\varphi_p)$. Moreover, in \cite{BR} it is proved that any Lorentzian $\mathcal{S}$-manifold with constant $\varphi$-sectional curvature is $\varphi$-null Osserman with respect to $(\xi_1)_p$, at any point, generalizing the similar known result for Lorentz Sasakian space forms.

Let $(M,\varphi,\xi_\alpha,\eta^\alpha,g)$ be a Lorentz $g.f.f$-manifold, with $\dim(M)=2n+s$, $s\geqslant1$. Fix $p\in M$ and consider $u\in N_\varphi((\xi_1)_p)$. Writing $u=(\xi_1)_p+x$, with $x\in S_\varphi((\xi_1)_p)$, we can consider the Jacobi operator $R_x\colon x^\bot\rightarrow x^\bot$ corresponding to $\bar R_u\colon\bar u^\bot\rightarrow\bar u^\bot$, and vice-versa. We are going to find out the link between these two operators.

\begin{proposition}\label{prop:01}
Let $(M,\varphi,\xi_\alpha,\eta^\alpha,g)$ be a Lorentz $\mathcal{S}$-manifold, with $\dim(M)=2n+s$, $s\geqslant1$, and $p\in M$. Then, $M$ is $\varphi$-null Osserman with respect to $(\xi_1)_p$ if, and only if, the eigenvalues of $R_x$ with their multiplicities are independent of $x\in S_\varphi((\xi_1)_p)$.
\end{proposition}
\proof Throughout the proof we fix $p\in M$ and, for simplicity, we omit it from the notations. Let us first suppose $s\geqslant 2$. Choose $u\in N_\varphi(\xi_1)$ and put $u=\xi_1+x$, with $x\in S_\varphi(\xi_1)$. If $V=x^\bot\cap\mathrm{Im}(\varphi)$, then $W_1=V\oplus\mathrm{span}(\xi_2,\ldots,\xi_s)$ is a geometrical realization of $\bar u^\bot$ and $W_2=V\oplus\mathrm{span}(\xi_1,\xi_2,\ldots,\xi_s)=x^\bot$. Identifying $\bar u^\bot\cong W_1$, and putting $\bar w=\pi(w)$, $w\in W_1$, from (\ref{eq:05}) and (\ref{eq:06}), for any $y\in V$ and any $\alpha,\beta\in\{2,\ldots,s\}$, we get
\begin{align}
&\bar g(\bar R_u(\bar \xi_\alpha),\bar \xi_\beta)=g(R_x(\xi_\alpha),\xi_\beta)=1, \label{eq:07}\\
&\bar g(\bar R_u(\bar y),\bar \xi_\beta)=g(R_x(y),\xi_\beta)=0. \label{eq:08}
\end{align}
Using (\ref{eq:04}), for any $y,z\in V$, one gets
\begin{equation}\label{eq:14}
\begin{split}
\bar g(\bar R_u(\bar y),\bar z)&=R(u,y,u,z)\\
                               &=R(\xi_1,y,\xi_1,z)+R(x,y,x,z)+R(\xi_1,y,x,z)+R(x,y,\xi_1,z)\\
                               &=g(y,z)+g(R_x(y),z).
\end{split}
\end{equation}
Finally, for any $y\in V$ and any $\alpha\in\{1,\ldots,s\}$:
\begin{equation}\label{eq:09}
g(R_x(y),\xi_1)=0, \qquad g(R_x(\xi_1),\xi_\alpha)=\left\{\begin{array}{cc}1&\quad \alpha=1\\ -1 & \quad \alpha\neq 1\end{array}\right.
\end{equation}
It follows that $V$ is an invariant subspace with respect to the action of both $\bar R_u$ and $R_x$. Choosing any orthonormal base $\mathcal{B}$ for $V$, then $\mathcal{B}_1=\mathcal{B}\cup\{\xi_2,\ldots,\xi_s\}$ and $\mathcal{B}_2=\mathcal{B}\cup\{\xi_1,\xi_2,\ldots,\xi_s\}$ are orthonormal basis for $W_1$ and $x^\bot$, respectively. If we denote by $A$ the $(2n-1)$-square matrix of $\bar R_u|_V$ with respect to the base $\mathcal{B}$, then (\ref{eq:07}), (\ref{eq:08}), (\ref{eq:14}) and (\ref{eq:09}) imply that
\[
C=\left(\begin{array}{cc}
A&0\\
0&B
\end{array}\right)
\]
is the matrix of $\bar R_u$ with respect to the bases $\bar{\mathcal{B}}_1=\pi(\mathcal{B}_1)$ of $\bar u^\bot$, and
\[
D=\left(
\begin{array}{ccc}
A-I_{2n-1}&\begin{array}{c}0\\\vdots\\0\end{array}&0\\
0\ldots0&-1&1\ldots1\\
0&\begin{array}{c}-1\\\vdots\\-1\end{array}&B
\end{array}
\right)
\]
is the matrix of $R_x$ with respect to $\mathcal{B}_2$, where, in both matrices, $B$ is the $(s-1)$-square matrix with all elements equal to $1$. Computing the characteristic polynomials of $C$ and $D$, we get
\[
\begin{split}
&p_C(\lambda)=p_A(\lambda)(-1)^{s-1}\lambda^{s-2}(\lambda-s+1),\\ &p_D(\lambda)=p_A(\lambda+1)(-1)^s\lambda^{s-1}(\lambda-(s-2)),
\end{split}
\]
from which the statement follows.

The same proof also works in the case $s=1$, with only straightforward modifications. Namely, for $s=1$, the matrix $B$ disappears, hence $C=A$ and
\[
D=\left(\begin{array}{cc}
	A-I_{2n-1}&0\\
	0&-1
\end{array}\right),
\]
thus obtaining $p_C(\lambda)=p_A(\lambda)$ and $p_D(\lambda)=-(\lambda+1)p_A(\lambda+1)$, from which the statement again follows. \endproof

\begin{remark}
Let $(M,\varphi,\xi_\alpha,\eta^\alpha,g)$ be a Lorentzian $\mathcal{S}$-manifold. Since it is clear that $S_\varphi((\xi_1)_p)=S_\varphi(-(\xi_1)_p)$, $p\in M$, from the above proposition we get that $M$ is $\varphi$-null Osserman with respect to $(\xi_1)_p$ if and only if so it is with respect to $-(\xi_1)_p$. Therefore, from now on, any Lorentzian $\mathcal{S}$-manifold satisfying the $\varphi$-null Osserman condition with respect to $(\xi_1)_p$ will be simply said to be \emph{$\varphi$-null Osserman at the point $p$}.
\end{remark}

Let $(M,\varphi,\xi_\alpha,\eta^\alpha,g)$ be a Lorentz $\mathcal{S}$-manifold, with $\dim(M)=2n+s$, $s\geqslant1$. Fix $p\in M$, consider $u\in N_\varphi((\xi_1)_p)$ and put $\psi(u)=u-(\xi_1)_p=x\in S_\varphi((\xi_1)_p)$.

As before, if $s\geqslant2$, putting $V=x^\bot\cap\mathrm{Im}(\varphi_p)$ and $U=\mathrm{span}((\xi_2)_p,\ldots,(\xi_s)_p)$, we can consider the geometrical realization $W=V\oplus U$ of $\bar u^\bot$, and identify $\bar u^\bot\cong W$. From the proof of Proposition \ref{prop:01}, it is clear that we can decompose $\bar R_u=\bar R_u|_V\circ p_V+\bar R_u|_U\circ p_U$, where $p_V$ and $p_U$ are the projections of $W$ onto $V$ and $U$, respectively. Clearly, $\bar R_u|_U\circ p_U$ always admits the eigenvalues $\lambda_0=0$ and $\lambda_1=s-1$, with multiplicity $m_0=s-2$ and $m_1=1$, independent of $u\in N_\varphi((\xi_1)_p)$. If $s=1$, the subspace $U$ simply disappears.

In any case, we can fix our attention only on the behaviour of the endomorphism $\bar R_u|_V\circ p_V$, which we denote, from now on, by $\bar R_u^\varphi$.

\begin{proposition}\label{prop:04}
Let $(M,\varphi,\xi_\alpha,\eta^\alpha,g)$ be a Lorentzian $\mathcal{S}$-manifold, with $\dim(M)=(4m+2)+s$, $s\geqslant1$, where $\dim(\mathrm{Im}(\varphi))=4m+2$. If $M$ is $\varphi$-null Osserman at $p\in M$, then for all $u\in N_{\varphi}((\xi_1)_p)$ only one of the following two cases can occur:
\begin{enumerate}
	\item [(i)] $\bar R_u^\varphi$ admits exactly one eigenvalue $c$ with multiplicity $4m+1$;
	\item [(ii)] $\bar R_u^\varphi$ admits exactly two eigenvalues $c_1$ and $c_2$ with multiplicities $1$ and $4m$.
\end{enumerate}
\end{proposition}
\proof Fix $p\in M$. Identifying $\mathrm{Im}(\varphi_p)\cong\R^{4m+2}$, we consider $S_{\varphi}((\xi_1)_p)\cong\Sf^{4m+1}$ and we endow $N_{\varphi}((\xi)_p)$ with the smooth structure such that $\psi\colon N_{\varphi}((\xi)_p)\rightarrow \Sf^{4m+1}$ is a diffeomorphism. Hence, for any $u\in N_\varphi((\xi_1)_p)$, putting $x=\psi(u)$, we can identify $V=x^\bot\cap\mathrm{Im}(\varphi_p)\cong T_x\Sf^{4m+1}$ and, under this identification, $\bar R_u^\varphi$ induces a unique endomorphism $\mathcal{R}_x\colon T_x\Sf^{4m+1}\rightarrow T_x\Sf^{4m+1}$ defined by $\mathcal{R}_x(y)=\pi^{-1}(\bar R_u^\varphi(\bar y))$, for all $y\in V$. It is clear that $\mathcal{R}_x$ and $\bar R_u^\varphi$ admit the same eigenvalues, and as we let $u$ vary in $N_\varphi((\xi_1)_p)$ we obtain a bundle homomorphism $\mathcal{R}=(\mathcal{R}_x)_{x\in\Sf^{4m+1}}\colon T\Sf^{4m+1}\rightarrow T\Sf^{4m+1}$. Let $c\in\R$ be an eigenvalue of $\bar R_u^\varphi$, independent of $u\in N_{\varphi}((\xi)_p)$. Putting $\mathcal{D}_x=\ker(\mathcal{R}_x-cI)$, we get a distribution $\mathcal{D}=(\mathcal{D}_x)_{x\in\Sf^{4m+1}}$ on $\Sf^{4m+1}$. By a well-known result about the maximal dimensions of distributions on spheres (see \cite{S}, pag.~155), the only possibilities for $\dim(\mathcal{D})$ are $1$, $4m$ or $4m+1$. If $\dim(\mathcal{D})=4m+1$, then $c$ is the only eigenvalue of each $\bar R_u^\varphi$, with multiplicity $4m+1$, and $(i)$ follows. If $\dim(\mathcal{D})=1$, for another eigenvalue $c'$ with distribution $\mathcal{D}'=(\mathcal{D}'_x)_{x\in\Sf^{4m+1}}$, $\mathcal{D}'_x=\ker(\mathcal{R}_x-c'I)$, we must have  $\dim(\mathcal{D}')=4m$, since otherwise we get a $2$-dimensional distribution $\mathcal{D}\oplus\mathcal{D}'$ on $\Sf^{4m+1}$, and $(ii)$ follows. \endproof

Note that in case $(i)$, $M$ has constant $\varphi$-sectional curvature at the point $p$, and if we suppose that $M$ is $\varphi$-null Osserman at each $p\in M$, and that the eigenvalues of $\bar R_u$ do not depend on the point $p$, then $M$ is a Lorentz $\mathcal{S}$-space form.

\section{The curvature of some $\varphi$-null Osserman Lorentzian $\mathcal{S}$-manifolds.}
In this section we study in full generality the curvature tensor of those $\varphi$-null Osserman Lorentzian $\mathcal{S}$-manifolds whose Jacobi operators $\bar R_u^\varphi$ admit exactly two distinct eigenvalues, thus including the case $(ii)$ of Proposition \ref{prop:04}. In \cite{BR} it has been proved a curvature result for the special subclass made of the $\varphi$-null Osserman Lorentzian $\mathcal{S}$-manifolds with two characteristic vector fields. Here, we provide an extension of that result to the general case of $\varphi$-null Osserman Lorentzian $\mathcal{S}$-manifolds with an arbitrary number of characteristic vector fields.

The key point of the proof is the special algebraic characterization for curvature-like maps provided by Lemma \ref{Lemma:01}, which generalizes the analogue result of \cite{BR}. We begin by finding, in the following remark, useful expressions for lightlike vectors which take advantage of the presence of the Lorentzian $\mathcal{S}$-structure.

Throughout what follows, we fix a point $p$ in the manifold and, for simplicity, we omit it from the notations.

\begin{remark}
Let $(M,\varphi,\xi_\alpha,\eta^\alpha,g)$ be a Lorentzian $\mathcal{S}$-manifold, $\dim(M)=2n+s$, $s\geqslant2$, and $u$ a lightlike tangent vector at $p$. Since $T_pM=\mathrm{Im}(\varphi)\oplus\mathrm{span}(\xi_1,\ldots\xi_s)$ and $\xi_1$ is unit timelike, we can always express the vector $u$, up to a non-vanishing factor, either in the form
\begin{equation}\label{1}
u=tx+\xi_1+\sum_{\alpha=2}^sk_\alpha\xi_\alpha,
\end{equation}
where $x\in\mathrm{Im}(\varphi)$ with $g(x,x)=1$, and $t,k_2,\ldots,k_s\in\mathbb{R}$ such that $t\neq0$ and  $t^2+\sum_{\alpha=2}^s(k_\alpha)^2=1$, or in the form
\begin{equation}\label{2}
	u=\xi_1+\sum_{\alpha=2}^sh_\alpha\xi_\alpha,
\end{equation}
where $h_2,\ldots,h_s\in\mathbb{R}$ with $\sum_{\alpha=2}^s(h_\alpha)^2=1$. 

A lightlike vector $u\in T_pM$ will be called \emph{of the first (resp.\ second) kind}, if it has the form \eqref{1} (resp.\ \eqref{2}). Any degenerate plane $\pi$ in $T_pM$ can be written in the form $\pi=\mathrm{span}(u,y)$, where $u\in T_pM$ is a lightlike vector either of the first or of the second kind, and $y\in u^\bot$. Let us see how to express $u^\bot$, according to the kind of $u$.

Suppose $u$ is of the first kind and put $w_0=\xi_1+\sum_{\alpha=2}^sk_\alpha\xi_\alpha$. Since $g(w_0,w_0)<0$, we can consider an orthonormal base $\mathcal{B}=(w_1,w_2,\ldots,w_{s-1})$ of the Euclidean space $E=w_0^\bot\cap\mathrm{span}(\xi_1,\ldots,\xi_s)$ and we have
\[
u^\bot=\mathrm{span}(u,\varphi x,x_2,\varphi x_2,\ldots,x_n,\varphi x_n,w_1,\ldots,w_{s-1}),
\]
where $(x,\varphi x,x_2,\varphi x_2,\ldots,x_n,\varphi x_n)$ is any orthonormal base of $\mathrm{Im}(\varphi)$. Hence, any $y\in u^\bot$ can be written as
\begin{equation}\label{y1}
y=au+by'+\lambda^iw_i,
\end{equation}
where $a,b,\lambda^1,\ldots,\lambda^{s-1}\in\mathbb{R}$, and $y'\in\mathrm{span}(\varphi x,x_2,\varphi x_2,\ldots,x_n,\varphi x_n)$, $g(y',y')=1$.

Suppose $u$ is of the second kind and put $\xi'=\sum_{\alpha=2}^sh_\alpha\xi_\alpha$. We can consider any orthonormal base $\mathcal{B}'=(w_1',\ldots,w_{s-2}')$ of $\xi'^\bot\cap\mathrm{span}(\xi_2,\ldots,\xi_s)$ and we get 
\[
u^\bot=\mathrm{span}(u)\oplus\mathrm{Im}(\varphi)\oplus\mathrm{span}(w_1',\ldots,w_{s-2}').
\]
Hence, any $y\in u^\bot$ can be written as
\begin{equation}\label{y2}
y=au+by'+\mu^iw_i',
\end{equation}
with $a,b,\mu^1,\ldots,\mu^{s-2}\in\R$ and $y'\in\mathrm{Im}(\varphi)$, $g(y',y')=1$.
\end{remark}

Let $(M,\varphi,\xi_\alpha,\eta^\alpha,g)$ be a Lorentzian $\mathcal{S}$-manifold, $\dim(M)=2n+s$, $s\geqslant1$, and let us consider the $(1,3)$-type algebraic curvature tensors $T$ and $S$ on $M$ defined by
\begin{align*}
	T(X,Y)Z&=g(\varphi Y,\varphi Z)\tilde\eta(X)\tilde\xi-g(\varphi X,\varphi Z)\tilde\eta(Y)\tilde\xi\\
	       &\quad-\tilde\eta(Y)\tilde\eta(Z)\varphi^2 X+\tilde\eta(X)\tilde\eta(Z)\varphi^2 Y,\\
	S(X,Y)Z&=g(\varphi X,\varphi Z)\varphi^2 Y-g(\varphi Y,\varphi Z)\varphi^2 X,
\end{align*}
for any $X,Y,Z\in\Gamma(TM)$.

Now, we are in a position to prove the following remarkable result which characterizes a special class of curvature-like maps on a Lorentzian $g.f.f$-manifold by the behaviour with respect to suitable degenerate planes, adapted to the study of the $\varphi$-null Osserman condition.

\begin{lemma}\label{Lemma:01}
Let $(M,\varphi,\xi_\alpha,\eta^\alpha,g)$ be a Lorentzian $g.f.f$-manifold, $\dim(M)=2n+s$, $s\geqslant1$. Fix $p\in M$, and let $F\colon(T_p M)^4\rightarrow \mathbb{R}$ be a curvature-like map satisfying
\begin{equation}\label{eq:18}
F(x,\xi_\alpha,y,\xi_\beta)=\varepsilon_\alpha\varepsilon_\beta g(\varphi x,\varphi y)\qquad\text{and}\qquad F(\varphi x,\varphi y, \varphi z,\xi_\alpha)=0,
\end{equation}
for any $x,y,z\in T_p M$ and $\alpha,\beta\in\{1,\ldots,s\}$. The following statements are equivalent.
\begin{itemize}
	\item [(a)] $F(u,y,u,y)=0$ on any degenerate $2$-plane $\pi=span\{u,y\}$ with $u\in N_\varphi((\xi_1)_p)$ and $y\in u^\bot\cap\mathrm{Im}(\varphi_p)$;
	\item [(b)] $F(x,y,z,w)=g_p(S_p(x,y)z,w)-g_p(T_p(x,y)z,w)$, for any $x,y,v,w\in T_p M$.
\end{itemize}
\end{lemma}
\begin{proof} Fix $p\in M$ and, for simplicity, let us omit it from the notations. We treat separately the case $s=1$ and $s\geqslant2$.

If $s=1$, $M$ reduces to a Lorentzian almost contact metric manifold $(M,\varphi,\xi,\eta,g)$, with $g(\xi,\xi)=-1$. Hence, $\tilde\eta=-\eta$, $\tilde\xi=\xi$ and, by \eqref{eq:01}, $g(\varphi X,\varphi Y)=g(X,Y)+\eta(X)\eta(Y)$, for any $X,Y\in\Gamma(TM)$. Using this, it is easy to see that, for any $x,y,z\in T_pM$
\begin{align*}
S(x,y)z&=g(y,z)x-g(y,z)\eta(x)\xi+\eta(y)\eta(z)x\\
       &\qquad-g(x,z)y+g(x,z)\eta(y)\xi-\eta(x)\eta(z)y,\\
T(x,y)z&=-g(y,z)\eta(x)\xi+g(x,z)\eta(y)\xi+\eta(y)\eta(z)x-\eta(x)\eta(z)y,
\end{align*}
from which it follows that the statement $(b)$ is equivalent to
\begin{equation}\label{eq:b'}
F(x,y,z,w)=g(x,w)g(y,z)-g(y,w)g(x,z).
\end{equation}
On the other hand, it is clear that, up to a non-vanishing multiplicative factor, any lightlike vector $u\in T_pM$ can be written as $u=\xi\pm x$, with $x\in\mathrm{Im}(\varphi)$, $g(x,x)=1$. This is equivalent to saying that $N(\xi)=N_\varphi(\xi)$ is the set of all the lightlike vector of $T_pM$. Therefore, any degenerate plane $\pi$ in $T_pM$ is spanned by $u\in N(\xi)$ and $y\in u^\bot\cap\mathrm{Im}(\varphi)$, hence the statement $(a)$ is equivalent to requiring that $F$ vanishes on any degenerate plane in $T_pM$. By Lemma \ref{degenere}, this is equivalent to requiring that there exists $k\in\R$ such that
\begin{equation}\label{eq:a'}
F(x,y,z,w)=k\{g(y,w)g(x,z)-g(x,w)g(y,z)\},
\end{equation}
for any $x,y,z\in T_pM$. The above formula and \eqref{eq:18} yield $-k=F(x,\xi,x,\xi)=1$, for any unit $x\in\mathrm{Im}(\varphi)$. Thus, \eqref{eq:a'} is equivalent to \eqref{eq:b'}, that is $(a)\Leftrightarrow(b)$.

Now, suppose $s\geqslant2$. It is obvious that $(b)$ implies $(a)$, since 
\begin{align*}
	g(T(u,y)u,y)&=\tilde\eta(u)\tilde\eta(u)g(\varphi^2 y,y)=-g(y,y),\\
	g(S(u,y)u,y)&=g(\varphi u,\varphi u)g(\varphi^2 y,y)=-g(y,y),
\end{align*}
for any $u\in N_\varphi(\xi_1)$ and $y\in\mathrm{Im}(\varphi)$.

Conversely, assume that $(a)$ holds and let $H\colon(T_pM)^4\rightarrow\R$ be the curvature-like map such that, for any $x,y,z,w\in T_p M$
\begin{equation}\label{Hdef}
H(x,y,z,w)=F(x,y,z,w)-g(S(x,y,v),w)+g(T(x,y,v),w).
\end{equation}
We are going to see that $H=0$, by using Lemma \ref{degenere}. Therefore, we have to calculate $H$ on any degenerate plane $\pi=\mathrm{span}(u,y)$, the lightlike vector $u$ being of either the first kind, or the second kind, and $y\in u^\bot$. First, note that \eqref{eq:18} clearly implies that 
\begin{equation}\label{eq:19}
F(x,v',y,v'')=\tilde\eta(v')\tilde\eta(v'')g(\varphi x,\varphi y)\qquad\text{and}\qquad F(x,\xi_\alpha,\xi_\beta,\xi_\gamma)=0,
\end{equation}
for any $x,y\in T_pM$, any $v',v''\in\mathrm{span}(\xi_1,\ldots,\xi_s)$ and any $\alpha,\beta,\gamma\in\{1,\ldots,s\}$. Now, suppose that $u$ is of the first kind and that $y\in u^\bot$ is given by \eqref{y1}. Using \eqref{eq:19}, we have
\begin{equation}\label{H1}
\begin{split}
F(u,y,u,y)&=b^2F(u,y',u,y')+\lambda^i\lambda^jF(u,w_i,u,w_j)+2\lambda^ibF(u,y',u,w_i)\\
	        &=b^2F(u,y',u,y')+\lambda^i\lambda^j\tilde\eta(w_i)\tilde\eta(w_j)g(\varphi u,\varphi u)\\
	        &=b^2F(u,y',u,y')+g(\varphi u,\varphi u)(\tilde\eta(y)-a\tilde\eta(u))^2,
\end{split}
\end{equation}
since, by \eqref{eq:18} and the First Bianchi Identity, $F(u,y',u,w_i)=0$. Now, using (\ref{1}) and \eqref{eq:19}, we get
\begin{align*}
	F(u,y',u,y')&=t^2F(x,y',x,y')+F(w_0,y',w_0,y')\\
	            &=g(\varphi u, \varphi u)F(x,y',x,y')+F(w_0,y',w_0,y')\\
	            &=g(\varphi u, \varphi u)F(x,y',x,y')+\tilde\eta(w_0)^2g(y',y')\\
	            &=g(\varphi u, \varphi u)F(x,y',x,y')+\tilde\eta(u)^2g(y',y').
\end{align*}
Substituting the above formula in (\ref{H1}), we get
\begin{align*}
	F(u,y,u,y)&=b^2g(\varphi u, \varphi u)F(x,y',x,y')+b^2\tilde\eta(u)^2g(y',y')\\
	          &\quad+g(\varphi u,\varphi u)(\tilde\eta(y)-a\tilde\eta(u))^2.
\end{align*}
Let us now put $u'=x+\xi_1$. We see that $u'\in N_\varphi(\xi_1)$ and that $y'\in\mathrm{Im}(\varphi)\cap u'^\bot$. Since $F(u',y',u',y')=F(x,y',x,y')+g(y',y')$, from the above formula we obtain
\begin{equation}\label{F1}
\begin{split}
	F(u,y,u,y)&=b^2g(\varphi u, \varphi u)F(u',y',u',y')-b^2g(y',y')g(\varphi u, \varphi u)\\
	          &\quad+b^2\tilde\eta(u)^2g(y',y')+g(\varphi u,\varphi u)(\tilde\eta(y)-a\tilde\eta(u))^2.
\end{split}
\end{equation}
Expanding $g(S(u,y)u,y)$, we get
\begin{equation}\label{TS1}
\begin{split}
	g(S(u,y)u,y)&=g(\varphi u, \varphi u)\{a^2 g(\varphi^2u,u)+b^2g(\varphi^2y',y')\}\\
	            &\qquad-a^2 g(\varphi u,\varphi u)g(\varphi^2 u,u)\\
	            &=-b^2 g(\varphi u,\varphi u)g(y',y').
\end{split}
\end{equation}
Analogously, expanding $g(T(u,y)u,y)$, we have
\begin{equation}\label{TS2}
\begin{split}
	g(T(u,y)u,y)&=-\tilde\eta(u)^2\{a^2g(\varphi u, \varphi u)+b^2g(\varphi y',\varphi y')\}\\
	            &\qquad+2a g(\varphi u,\varphi u)\tilde\eta(u)\tilde\eta(y)-g(\varphi u,\varphi u)\tilde\eta(y)^2\\
	            &=-b^2g(y',y')\tilde\eta(u)^2\\
	            &\qquad-g(\varphi u,\varphi u)\{a^2\tilde\eta(u)^2-2a\tilde\eta(u)\tilde\eta(y)+\tilde\eta(y)^2\}\\
	            &=-b^2g(y',y')\tilde\eta(u)^2-g(\varphi u,\varphi u)(a\tilde\eta(u) -\tilde\eta(y))^2.
\end{split}
\end{equation}
Now \eqref{F1}, \eqref{TS1} and \eqref{TS2} imply $H(u,y,u,y)=b^2g(\varphi u, \varphi u)F(u',y',u',y')$, where $u'\in N_\varphi(\xi_1)$ and $y'\in u'^\bot\cap\mathrm{Im}\varphi$. Thus $(a)$ yields $H(u,y,u,y)=0$ when $u$ is a lightlike vector of the first kind.

Now, suppose that $u$ is a lightlike vector of the second kind and that $y\in u^\bot$ is given by \eqref{y2}. Using \eqref{2}, with analogous computations as above, we find
\begin{align*}
F(u,y,u,y)&=b^2F(u,y',u,y')+\mu^i\mu^jF(u,w'_i,u,w_j')+2\mu^ibF(u,y',u,w_i')\\
          &=b^2g(y',y')\tilde\eta(u)^2.
\end{align*}
Moreover, $\varphi u=0$, hence $g(S(u,y)u,y)=0$ and, with straightforward calculations,
\[
g(T(u,y)u,y)=\tilde\eta(u)\tilde\eta(u)g(\varphi^2 y, y)=-b^2g(y',y')\tilde\eta(u)^2.
\]
The above identities imply $H(u,y,u,y)=0$ when $u$ is a lightlike vector of the second kind.

Thus, $H$ vanishes on any degenerate $2$-plane. By Lemma \ref{degenere}, there exists $k\in\R$ such that $H(x,y,z,w)=k\{g(y,z)g(x,w)-g(x,z)g(y,w)\}$, for any $x,y,z,w\in T_pM$. It follows that $H(x,\xi_\alpha,x,\xi_\alpha)=-k\varepsilon_\alpha$, for any unit $x\in \mathrm{Im}(\varphi)$ and any $\alpha\in\{1,\ldots,s\}$. On the other hand, \eqref{Hdef} yields $H(x,\xi_\alpha,x,\xi_\alpha)=0$, hence $H(x,y,z,w)=0$, for any $x,y,z,w\in T_p M$, and $(b)$ follows. \end{proof}

\begin{remark}
It is worth noting that the analogous result stated in \cite{BR} in the case of a Lorentzian $\mathcal{S}$-manifold with exactly two characteristic vector fields is a particular case of the above lemma. Indeed, using the notations of \cite{BR}, the tensor $S_*$ is nothing but $S$, and a long, but straightforward calculation shows that the tensor $T$ reduces to $S^*$, when the manifold has only two characteristic vector fields.
\end{remark}

From \cite{GSV} it is well-known that, on an even-dimensional Osserman Riemannian manifold $(M,g)$, whose Jacobi operators $R_x$, $x\in S(M)$ admit exactly two eigenvalues $c_1,c_2\in\R$, with multiplicities $1$ and $\dim(M)-2$, it is possible to define an almost Hermitian structure $J$ by associating to each $x\in S(M)$ the unit eigenvector $Jx$ of $R_x$ with respect to the eigenvalue $c_1$ with multiplicity $1$ (see also \cite{GKV}). In \cite[Lemma 3.5.1, p.\ 202]{G} one can find the above result stated in a purely algebraic setting.

In \cite{BR}, the construction has been adapted to the case of a $(2n+s)$-dimensional $\varphi$-null Osserman Lorentzian $\mathcal{S}$-manifold $(M,\varphi,\xi_\alpha,\eta^\alpha,g)$ at a point $p\in M$, with Jacobi operators $\bar R_u^\varphi$, $u\in N_\varphi((\xi_1)_p)$, admitting eigenvalues as above. In this way, an almost Hermitian structure is obtained on $\mathrm{Im}(\varphi_p)$.

Now, let us use that construction here, and introduce the $(1,3)$-type tensor $R^0$ and $R^J$ on $T_pM$, defined by
\begin{align*}
R^0(x,y)z&=g(\varphi^2y,\varphi^2z)\varphi^2x-g(\varphi^2x,\varphi^2z)\varphi^2y,\\
R^J(x,y)z&=g(J\varphi^2y,\varphi^2z)J\varphi^2x-g(J\varphi^2x,\varphi^2z)J\varphi^2y+2g(\varphi^2x,J\varphi^2y)J\varphi^2z
\end{align*}
for any $x,y,z\in T_pM$. Using exactly the same proof of \cite{BR}, except for the use of Lemma \ref{Lemma:01}, of course, and using the properties of Lemma \ref{lem:99}, we obtain the following curvature characterization.

\begin{theorem}
Let $(M,\varphi,\xi_\alpha,\eta^\alpha,g)$ be a Lorentzian $\mathcal{S}$-manifold, $\dim(M)=2n+s$, $n>1$, $s\geqslant1$. The following statements are equivalent.
\begin{itemize}
	\item [(a)] $M$ is $\varphi$-null Osserman at a point $p$ and for any $u\in N_\varphi((\xi_1)_p)$ the Jacobi operators $\bar R_u^\varphi$ admit exactly two distinct eigenvalues $c_1,c_2\in\R$, with multiplicities $1$ and $2n-2$, respectively. 
	\item [(b)] There exists an almost Hermitian structure $J$ on $\mathrm{Im}(\varphi_p)$ and  $c_1,c_2\in\R$, such that
	\begin{equation*}
	R_p(x,y)z=T_p(x,y)z-S_p(x,y)x-c_2R^0(x,y)z-\frac{c_1-c_2}{3}\,R^J(x,y)z,
	\end{equation*}
	for any $x,y,z\in T_p M$.
\end{itemize}
\end{theorem} 

\begin{remark}
In particular, for $s=1$, when the $\varphi$-null Osserman condition at a point is nothing but the null Osserman condition on a Lorentzian almost contact manifold, we have seen that $T_p(x,y)z-S_p(x,y)z=g_p(x,z)y-g_p(y,z)x$, for any $x,y,z\in T_pM$. Using this, we obtain the following final result as a corollary of the above theorem.
\end{remark}

\begin{corollary}
Let $(M,\varphi,\xi,\eta,g)$ be a Lorentz Sasakian manifold, $\dim(M)=2n+1$, $n>1$. The following statements are equivalent.
\begin{itemize}
	\item [(a)] $M$ is null Osserman at a point $p$ and for any $u\in N_\varphi(\xi_p)$ the Jacobi operators $\bar R_u$ admit exactly two distinct eigenvalues $c_1,c_2\in\R$, with multiplicities $1$ and $2n-2$, respectively. 
	\item [(b)] There exists an almost Hermitian structure $J$ on $\mathrm{Im}(\varphi_p)$ and  $c_1,c_2\in\R$, such that
	\begin{equation*}
	R_p(x,y)z=g_p(x,z)y-g_p(y,z)x-c_2R^0(x,y)z-\frac{c_1-c_2}{3}\,R^J(x,y)z,
	\end{equation*}
	for any $x,y,z\in T_p M$.
\end{itemize}
\end{corollary}

\end{document}